\documentclass[12pt]{article}

\usepackage[inner=32mm,outer=31mm,tmargin=30mm,bmargin=40mm]{geometry}

\usepackage{latexsym,amsfonts,amsmath,amssymb,epsfig,tabularx,amsthm,dsfont,mathrsfs}

\usepackage{enumerate}

\usepackage{titling}
\usepackage{hyperref}
\hypersetup{colorlinks=true,
  linkcolor=black,
  citecolor=black,
  urlcolor=blue}

\theoremstyle{plain}

\newtheorem{thm}{Theorem}[section]
\newtheorem{lem}[thm]{Lemma}

\newtheorem{prop}[thm]{Proposition}

\theoremstyle{definition}
\newtheorem{rem}[thm]{Remark}
\newtheorem{example}[thm]{Example}

\renewcommand{\P}{{\mathbb P}}
\newcommand{\expect}{\operatorname{\mathbb{E}}}

\newcommand{\dint}{\,\mathup{d}}

\newcommand{\Finput}{\mathcal{F}}

\newcommand{\Ball}{\mathcal{B}}

\DeclareMathOperator{\supp}{supp}

\newcommand{\ind}{\mathds{1}}

\DeclareMathOperator{\Vol}{Vol}

\DeclareMathOperator{\diag}{diag}
\DeclareMathOperator*{\esssup}{ess\,sup}
\DeclareMathOperator*{\essinf}{ess\,inf}

\DeclareMathOperator{\median}{med}

\newcommand{\deter}{\textup{det}}

\newcommand{\mix}{\textup{mix}}
\newcommand{\Wmix}{\mathbf{W}}
\newcommand{\Wspace}{\mathcal{W}}
\newcommand{\MC}{\textup{MC}}

\newcommand{\ellmean}{\ell\textup{-mean}}
\newcommand{\oneMean}{1\textup{-mean}}
\newcommand{\twoMean}{2\textup{-mean}}

\newcommand{\prob}{\textup{prob}}

\DeclareMathOperator{\Int}{INT}

\newcommand{\eps}{\varepsilon}

\newcommand{\embed}{\hookrightarrow}
\renewcommand{\rho}{\varrho}

\newcommand{\veci}{\mathbf{i}}

\newcommand{\vecm}{\mathbf{m}}

\newcommand{\vecs}{\mathbf{s}}
\newcommand{\vecu}{\mathbf{u}}
\newcommand{\vecv}{\mathbf{v}}

\newcommand{\vecx}{\mathbf{x}}
\newcommand{\vecy}{\mathbf{y}}
\newcommand{\vecz}{\mathbf{z}}
\newcommand{\vecX}{\mathbf{X}}

\newcommand{\vecalpha}{\boldsymbol{\alpha}}

\newcommand{\R}{{\mathbb R}}

\newcommand{\N}{{\mathbb N}}
\newcommand{\Z}{{\mathbb Z}}

\DeclareMathAlphabet{\mathup}{OT1}{\familydefault}{m}{n}

\newcommand{\widebar}[1]{\mbox{\kern1.5pt\hbox{\vbox{\hrule height 0.6pt \kern0.35ex
        \hbox{\kern-0.15em \ensuremath{#1 }\kern0.0em}}}}\kern-0.1pt}

\newcommand{\fracts}[2]{{\textstyle\frac{#1}{#2}}}

\newlength{\fixboxwidth}
\title{Optimal confidence for Monte Carlo integration of smooth functions}

\date{\today}

\author{Robert J. Kunsch\thanks{Institut f\"ur Mathematik, 
Universit\"at Osnabr\"uck, Albrechtstr. 28a, 49076 Osnabr\"uck, Germany, 
Email: robert.kunsch@uni-osnabrueck.de},\; 
and Daniel Rudolf\thanks{Institute for Mathematical Stochastics, Universit\"at G\"ottingen 
\& Felix-Bernstein-Institute for Mathematical Statistics, 
Goldschmidtstra\ss e 7, 
37077 G\"ottingen, Germany, 
Email: daniel.rudolf@uni-goettingen.de}
}

\begin{document}

\maketitle

\begin{abstract}  
  We study the complexity of approximating integrals of smooth functions
  at absolute precision~$\eps > 0$ with confidence level $1 - \delta \in (0,1)$.
  The optimal
  error rate            
  for multivariate functions
  from classical isotropic Sobolev spaces $W_p^r(G)$
  with sufficient smoothness
  on bounded Lipschitz domains~$G \subset \R^d$
  is determined.
  It turns out that the integrability index~$p$
  has an effect on the influence of the uncertainty~$\delta$ in the complexity.
  In the limiting case~$p = 1$ we see
  that deterministic methods cannot be improved by randomization.
  In general, higher smoothness reduces the additional effort
  for diminishing the uncertainty.
  Finally, we add a discussion about this problem
  for function spaces with mixed smoothness.
\end{abstract}


\section{Introduction} \label{sec: intro}

We want to compute the integral
\begin{equation}   \label{eq:INT} 
  \Int (f) = \int_{G} f(\vecx) \dint \vecx
\end{equation} 
of $f\colon G\to \R$ from the unit ball $\Ball_\Wspace$ of a
(semi-)normed linear space $\Wspace$ of functions 
defined on a domain~$G \subset \R^d$
where we are only allowed to use 
function values as information within randomized algorithms.
The focus lies on the
\emph{$(\eps,\delta)$-complexity~$n^{\MC}_{\prob}(\eps,\delta,\Wspace)$},
that is, the minimal number~$n$ of function values needed
for randomized algorithms~$A_n$ 
in order to approximate the integral~\eqref{eq:INT} such that
\begin{equation} \label{eq:(eps,delta)-appr}
  \P\{|A_n(f) - \Int f| > \eps \} \leq \delta
  \qquad \text{for all $f$ with $\|f\|_{\Wspace}\leq 1$,}
\end{equation}
where $\|\cdot\|_{\Wspace}$ is the (semi-)norm of $\Wspace$.
A method with this property
of guaranteeing a small (absolute) \emph{error~$\eps > 0$}
with \emph{confidence $1-\delta \in (0,1)$} (or \emph{uncertainty~$\delta$})
for inputs from the unit ball $\Ball_\Wspace$
is called \emph{$(\eps,\delta)$-approximating in $\Wspace$},
see also \cite{KNR18}.
We also consider the
\emph{$n$-th minimal probabilistic Monte Carlo error at uncertainty~$\delta$}, 
defined by
\begin{equation} \label{eq:e^prob}
  e^{\MC}_{\prob}(n,\delta,\Wspace)
    := \inf\left\{\eps > 0 \mid
                  \exists \text{ $(\eps,\delta)$-approximating algorithm $A_n$
                  in $\Wspace$}
            \right\} \,.
\end{equation}
The probabilistic error criterion from above
is less common in \emph{information-based complexity} (IBC)
where the standard notion of \emph{Monte Carlo error}
is some type of mean error. In general,
for some $\ell\geq1$ the \emph{$n$-th minimal $\ell$-mean Monte Carlo error}
is given by
\begin{equation} \label{eq:ME}
  e^{\MC}_{\ellmean}(n,\Wspace)
    := \inf_{A_n} \sup_{\| f \|_{\Wspace} \leq 1}
          \left(\expect |A_n(f) - \Int f|^\ell\right)^{1/\ell} \,.
\end{equation}
Here, the infimum is taken over all randomized algorithms which use at most 
$n$ function values.
Most frequently studied are
the \emph{root mean squared error} (RMSE), that is, $e^{\MC}_{\twoMean}$,
as well as the expected error $e^{\MC}_{\oneMean}$. 
Accordingly,
we define the \emph{$\eps$-complexity} $n^{\MC}_{\ellmean}(\eps,\Wspace)$
as the minimal number of function values
needed by a randomized algorithm that guarantees
a worst case $\ell$-mean error smaller than~$\eps$.
For more details on IBC we refer to the books~\cite{NW08,NW10,NW12,TWW88}.

One might argue that the error criterion does not matter.
However, this is not true. For example, 
using Markov's inequality
it is always possible to construct $(\eps,\delta)$-approximating algorithms
once we know methods which work for arbitrarily small mean errors.
That way, however,
the cost estimates are not optimal in terms of the $\delta$-dependence,
namely polynomial rather than logarithmic.
In some situations this can be fixed by using more advanced inequalities
such as Hoeffding bounds.
In other situations commonly known algorithms may need to be modified
which lead to more robust methods less prone to outliers.
For this reason the probabilistic criterion is frequently used in statistics,
see for example \cite{GNP13,HJLO13,Hu17Bernoulli,Hu17}.
Furthermore, there are numerical problems
which can be solved with respect to the probabilistic $(\eps,\delta)$-criterion
but the mean error is unbounded, see~\cite{KNR18}.
In other words, this criterion
seems to be the right one
for the concept of solvability.

In Section~\ref{sec:LBs} we provide two generic lower bounds for
the $n$-th minimal probabilistic Monte Carlo error based on bump functions.
In Section~\ref{sec:UBs} we discuss several approaches for deriving upper error bounds 
on Sobolev classes and discuss 
in which cases they lead to optimal rates.
We mainly consider classical
isotropic Sobolev spaces~$W_p^r(G)$
on domains~$G \subseteq \R^d$.
For integer smoothness~$r \in \N_0$
and integrability parameter $1 \leq p \leq \infty$,
these spaces are given by 
\begin{equation*}
  W_p^r(G)
    :=\biggl\{f \in L_p(G) \,\bigg|\,
              \|f\|_{W_p^r(G)}
                := \biggl(\sum_{\substack{\vecalpha \in \N_0^d\\
                                   |\vecalpha|_1 \leq r}}
                     \|D^{\vecalpha} f\|_{L_p(G)}^p\biggr)^{1/p}
                < \infty
       \biggr\} \,,  
\end{equation*}
with the usual modification for~$p=\infty$ and the weak derivative
$D^{\vecalpha}f = \partial_{x_1}^{\alpha_1} \cdots \partial_{x_d}^{\alpha_d} f$
for multi-index $\vecalpha=(\alpha_1,\dots,\alpha_d)\in\N_0^d$.
Note that for~$r=0$ we obtain the Lebesgue spaces $L_p(G)$.

Our main result is for spaces~$W_p^{r}(G)$
on bounded Lipschitz domains~$G \subset \R^d$ (see~\cite{NT06} for a definition),
and with sufficient smoothness, $rp>d$.
In asymptotic notation (see definitions below) it states
\begin{equation} \label{eq:erate-intro}
  e^{\MC}_{\prob}(n,\delta,W_p^{r}(G))
    \asymp n^{-r/d} \, \min\left\{1, \,
                                  \left(\frac{\log \delta^{-1}}{n}\right)^{1-1/q}
                           \right\}
\end{equation}
with~$q := \min\{p,2\}$, or equivalently
\begin{equation} \label{eq:nrate-intro}
  n^{\MC}_{\prob}(\eps,\delta,W_p^{r}(G))
    \asymp \min\left\{\eps^{-d/r},\,
          \eps^{-\left.1\middle/
                  \left(\frac{r}{d} + \frac{q-1}{q}\right)\right.}
            \, (\log \delta^{-1})^{\left.1\middle/
                                    \left(\frac{q}{q-1} \cdot \frac{r}{d} + 1
                                    \right)\right.}
               \right\}\,,
\end{equation}
see Theorem~\ref{thm:LBs} and Theorem~\ref{thm:UBsWprSep}.
The condition~$r p > d$ guarantees that the space $W_p^r(G)$
is compactly embedded in the space of continuous functions,
see for instance~\cite{AF75/02}.
Only then function evaluations are well
defined and there exist deterministic integration methods,
which in this case provide error bounds with rate~$n^{-r/d}$.
These worst case bounds come into play
if we demand extremely high confidence~$1 - \delta$ close to~$1$.
It also turns out that for~$p=1$ the uncertainty~$\delta$ does not play any role,
which shows that deterministic methods are optimal in that case.
In the power of~$n$, we recover the well known gain of
$1-1/p$ for~$1 < p < 2$,
and $1/2$ for $p \geq 2$,
which Monte Carlo methods achieve compared to deterministic methods.
The influence of the uncertainty $\delta$ grows with the gain in the error rate.
In terms of the complexity~\eqref{eq:nrate-intro} we observe that
the higher the smoothness~$r$ the weaker the dependence on~$\delta$.

\textbf{Asymptotic notation:}
  For functions $e,f \colon \N\times (0,1) \to \R$
  we use the notation
  $e(n,\delta) \preceq f(n,\delta)$,
  meaning that there is some $n_0\in\N$ and $\delta_0\in(0,1)$ such that
  $e(n,\delta) \leq c f(n,\delta)$
  with some (possibly $(d,r)$-dependent) constant~$c > 0$
  for all $n \geq n_0$
  and $\delta \in (0,\delta_0)$.
  Sometimes we add the restriction $n \succeq \log \delta^{-1}$,
  then $e(n,\delta) \leq c f(n,\delta)$ is only meant to hold for
  $\delta \in (0,\delta_0)$ and $n \geq n_0 \log\delta^{-1}$.
  Similarly we denote asymptotics for complexity functions $n(\eps,\delta)$,
  describing a behaviour for small~$\eps,\delta > 0$.
  Asymptotic equivalence $e(n,\delta) \asymp f(n,\delta)$
  is a shorthand for \mbox{$e(n,\delta)\preceq f(n,\delta) \preceq e(n,\delta)$}.
  The notion $e(n,\delta) \prec f(n,\delta)$ means
  ``$e(n,\delta) \preceq f(n,\delta)$ but not $f(n,\delta) \preceq e(n,\delta)$''.

\section{Lower bounds}
\label{sec:LBs}

We start with the lower bounds
as these are easily obtained for the whole parameter range
of the function spaces we consider.

\subsection{Auxiliary lemmas} \label{sec: aux_lem}
As before, let $\Wspace$ be a space of functions defined on a domain~$G$,
equipped with a (semi-)norm $\|\cdot\|_{\Wspace}$.
An abstract Monte Carlo algorithm defined for such functions
is a family $A_n = (A_n^{\omega})_{\omega \in \Omega}$
of mappings
$A_n^{\omega} \colon
  \Wspace \xrightarrow{N^{\omega}} \R^n
          \xrightarrow{\phi^{\omega}} \R$,
          indexed with elements~$\omega$ from a probability space~$(\Omega,\Sigma,\P)$,
such that the error functional
\mbox{$\omega \mapsto |A_n^{\omega}(f) - \Int{f}|$} is measurable.
Here,
\begin{equation*}
  \vecy = N^{\omega}(f) = (f(\vecx_1^{\omega}),\ldots,f(\vecx_n^{\omega}))
\end{equation*}
is the information we collect about a problem instance~$f$,
from which the output $A_n^{\omega}(f) = \phi^{\omega}(\vecy)$ is generated.
One might consider adaptive strategies to acquire information,
that is, 
$\vecx_i^{\omega}$  
might depend on the
previously obtained information $y_1,\ldots,y_{i-1}$.
Our lower bounds do hold for this type of algorithms,
but the upper bounds we present are based on non-adaptive methods.
For simplicity, in this paper we restrict to methods
with fixed cardinality~$n$.
In general, the number of function values an algorithm collects
might be random and even depend on the input,
see for instance~\cite{GNP13,Hu17Bernoulli,KNR18}. Let us mention here that
our auxiliary lemmas on lower bounds, Lemma~\ref{lem:aux1} and \ref{lem:aux2},
would then still hold with slightly worse constants.

In the spirit of Bakhvalov~\cite{Bakh59},
for proving lower bounds
we switch to an \emph{average input setting}
with a discrete probability measure~$\mu$
supported within the input set---%
which in our case is the unit ball $\Ball_\Wspace$ of the space~$\Wspace$---%
and make use of the relation
\begin{align}
  \sup_{\|f\|_{\Wspace}\leq 1} \P\{|A_n(f) - \Int f| > \eps\}
    &\geq \int_{\Ball_\Wspace} \int_{\Omega}
              \ind_{\{|A_n^{\omega}(f) - \Int f| > \eps\}}
            \dint\P(\omega) \dint\mu(f)
       \nonumber\\
    &= \int_{\Omega} \int_{\Ball_\Wspace}
            \ind_{\{|A_n^{\omega}(f) - \Int f| > \eps\}}
          \dint\mu(f) \dint\P(\omega)
       \nonumber\\
    &\geq \inf_{Q_n} \mu\{f \colon |Q_n(f) - \Int f| > \eps\} \,,
    \label{eq:Bakh}
\end{align}
where the infimum is taken
over all \emph{deterministic} integration methods $Q_n$ that use $n$ function values.
(For fixed~$\omega$, the realisation $A_n^{\omega}$ of a given algorithm
can be regarded as a deterministic algorithm.)
In the proof of the lower bounds we use the implication
\begin{equation} \label{eq:delta<->eps}
  \sup_{\|f\|_{\Wspace} \leq 1} \P\{|A_n(f) - \Int(f)| > \eps\}
    > \delta
  \quad \Longrightarrow \quad
   e^{\MC}_{\prob}(A_n,\delta,\Wspace) \geq \eps \,,
\end{equation}
where $e^{\MC}_{\prob}(A_n,\delta,\Wspace)$ is the infimum of all~$\eps > 0$
such that the algorithm~$A_n$ is $(\eps,\delta)$-approximating in $\Wspace$.

Depending on the integrability index~$p$ of the Sobolev classes
we choose different probability measures~$\mu$
in order to obtain appropriate lower bounds.
Similarly to \cite[Proposition~1 and 2 in Section~2.2.4]{No88}
we have the following two generic lemmas,
now for the probabilistic instead of the root mean squared error.
The first one applies for integrability~$2 \leq p \leq \infty$.

\begin{lem}  \label{lem:aux1}
  For $n \geq 17$ 
  and a natural number $N \geq 5n + 6$,
  assume that there are functions \mbox{$f_i \colon G \to \R$},
  with~$i=1,\ldots,N$,
  satisfying the following conditions:
  \begin{enumerate}
    \item for $i = 1,\dots,N$,
      the sets~$G_i := \{\vecx \in G \colon f_i(\vecx) \not= 0\}$
      are pairwise disjoint,
      and $\Int(f_i) = \gamma$ for some $\gamma > 0$;
    \item for signs \mbox{$s_i \in \{\pm 1\}$},
      the function~\mbox{$f_{\vecs} := \sum_{i = 1}^N s_i \, f_i$}
      is an element of the input set $\Ball_\Wspace$,
      that is, $\|f_{\vecs}\|_{\Wspace} \leq 1$.
  \end{enumerate}
  Then, for any uncertainty level 
  \mbox{$0 < \delta < 1/3$},
  we have
  \begin{equation*}
    e^{\MC}_{\prob}(n,\delta,\Wspace)
      \geq \gamma \, \min\left\{
                        \,n^{1/2}
                         \,\sqrt{\log_4\frac{1}{3\delta}} ,\,
                      n
                \right\}\,.
  \end{equation*}
\end{lem}
\begin{proof}
  Let $\mu$ be the uniform distribution on the finite set
  \begin{equation*}
    \Finput
      := \left\{f_{\vecs} = \sum_{i=1}^N s_i f_i
                \colon
                  s_i \in \{\pm 1\}
          \right\}
      \subset \Ball_\Wspace \,.
  \end{equation*}
  Let~$Q_n \colon \Finput \to \R$ be a deterministic algorithm using
  $n$~function values.
  Without loss of generality, we may assume
  that the algorithm computes function values \mbox{$y_i = f(\vecx_i)$}
  with~\mbox{$\vecx_i \in G_i$} for~$i=1,\ldots,n$.
  Hence, from the $i$-th piece of information we learn
  whether~$s_i = +1$ or~$-1$ for $f = f_{\vecs}$.

  Note  that,
  given the information~$\vecy = N(f)$,
  there are still~$k := N-n$ unknown signs~$s_i$.
  The conditional distribution of 
  $\Int(f)$ given $\vecy$
  can be represented as the distribution of
  \begin{equation*}
    g_{\vecy} + \gamma X_k
  \end{equation*}
  where
  \begin{align*}
    g_{\vecy}
      &:= \gamma \, \sum_{i = 1}^n
    s_i \,, \qquad\text{and} &
    X_k
      &:= \sum_{i=1}^k Z_i
    \qquad\text{ with $Z_i \stackrel{\text{iid}}{\sim}
                  \operatorname{Rademacher}$.}
  \end{align*}
  Since this is the situation for all information
  $\vecy \,\widehat{=}\,  (s_1,\ldots,s_n)$,
  we obtain 
  \begin{align*}
    \nonumber\mu\{|Q_n(f) - \Int f| > \eps\}
      &\,\geq\, \inf_{a \in \R} \P\{\gamma \, |X_k - a| > \eps\}\\
    \nonumber
      &\,=\, \inf_{a \in \R}
                2^{-k} \sum_{j=0}^k
                  \binom{k}{j} \,
              \ind\{\gamma \, |2j-k - a| > \eps\}\,.
  \end{align*}
  At most \mbox{$k' := \lfloor \eps/\gamma \rfloor + 1$} terms are removed
  from the binomial sum, optimally the central ones,
  so we have
  \begin{align*}
  \nonumber\mu\{|Q_n(f) - \Int f| > \eps\} 
      &\,\geq\, 2^{-k}
              \Biggl[\sum_{j=0}^{\left\lfloor \frac{k-k'}{2} \right\rfloor}
                        \binom{k}{j}
                    + \sum_{j=\left\lceil \frac{k+k'+1}{2} \right\rceil}^k
                        \binom{k}{j}
              \Biggr].
  \end{align*}
  We employ Lemma~\ref{lem:BinomTail1} twice, namely,
  for odd $k$ 
  with 
  \begin{equation*}
    t = \lceil (k'-1)/2 \rceil
   \quad \text{and} \quad
   t = \lceil k'/2 \rceil \leq \eps/(2\gamma) + 1
  \end{equation*}
  as well as for even~$k$ with~
  \begin{equation*}
    t = \lceil k'/2 \rceil
   \quad \text{and} \quad
    t = \lceil (k'+1)/2 \rceil \leq \eps/(2\gamma) + 3/2.
  \end{equation*}
  In order to match the conditions of Lemma~\ref{lem:BinomTail1},
  we restrict to \mbox{$\eps/\gamma \leq (k-6)/4$}.
  Hence under that assumption we have
  \begin{align}
    \label{eq:|Qf-INTf|>eps|y,1}
    \mu\{|Q_n(f) - \Int f| > \eps\}
      &\,\geq\, \frac{1}{1 + 2/\sqrt{\pi}}
        \, \exp\left(- \frac{4 \, (\log 2) \, (\eps/\gamma + 2)^2}{k}
                \right)\,.
  \end{align}
  Note that~\mbox{$k = N - n \geq (5n + 6) - n = 4n + 6$},
  so then the condition \mbox{$\eps/\gamma \leq n$}
  is sufficient for \eqref{eq:|Qf-INTf|>eps|y,1} to hold.
        The right-hand side of \eqref{eq:|Qf-INTf|>eps|y,1}
  can be further simplified via
  \mbox{$(\eps/\gamma + 2)^2 \leq 2 \,(\eps^2/\gamma^2 + 4)$},
  exploiting~\mbox{$k > 4n$}
  and also \mbox{$n \geq 17$}. 
  For~\mbox{$0 < \eps \leq \gamma n$}
  this leads to
  \begin{equation} \label{eq:delta(2n,mu)>c*exp(-C*eps^2/n)}
    \mu\{|Q_n(f) - \Int f| > \eps\}
      > 
      \frac{2^{-8/n}}{1 + 2/\sqrt{\pi}}
         \, 2^{-2\,\eps^2 / (n \, \gamma^2)}
      > \frac{1}{3} \, 4^{-\eps^2 / (n \, \gamma^2)} \,.
  \end{equation}
  By Bakhvalov's trick~\eqref{eq:Bakh}
  this is a lower bound for the worst case uncertainty
  $\sup_{\|f\|_{\Wspace} \leq 1} \P\{|A_n(f) - \Int(f)| > \eps\}$,
  holding for any Monte Carlo algorithm~$A_n$.
  Regarding the right-hand side of~\eqref{eq:delta(2n,mu)>c*exp(-C*eps^2/n)}
  as a given $\delta$,
  the implication~\eqref{eq:delta<->eps} finally provides the assertion.
  Pay attention that for too small~$\delta$,
  namely $0 < \delta < \frac{1}{3} \, 4^{-n}$,
  isolating~$\eps$ in \eqref{eq:delta(2n,mu)>c*exp(-C*eps^2/n)}
  is misleading to delusive error bounds exceeding $\gamma n$
  which, however, violates the conditions on~$\eps$.
  In this case we can only conclude that $\eps = \gamma n$ is a lower bound
  for~$e^{\MC}_{\prob}(n,\delta,\Wspace)$.
\end{proof}

The following result will be useful for integrability $1 < p < 2$.

\begin{lem} \label{lem:aux2}
  For $n \in \N$ and a natural number $N \geq 4 n$,
  assume that there are functions \mbox{$f_i \colon G \to \R$},
  with~$i=1,\ldots,N$,
  satisfying the following conditions:
  \begin{enumerate}
    \item for $i = 1,\dots,N$,
      the sets~$G_i := \{\vecx \in G \colon f_i(\vecx) \not= 0\}$
      are pairwise disjoint,
      and $\Int(f_i) = \gamma$ for some $\gamma > 0$;
    \item for $I \subset \{1,\ldots,N\}$ with $\# I = M$
      for some given natural number $M \leq N$,
      and for signs \mbox{$s_i \in \{\pm 1\}$},
      the function~\mbox{$f_{I,\vecs} := \sum_{i \in I} s_i \, f_i$}
      is an element
      of the input set~$\Ball_\Wspace$,
      that is, $\|f_{I,\vecs}\|_{\Wspace} \leq 1$.
  \end{enumerate}
  Then, for any~$0 < \delta < \frac{1}{2} \, 2^{-\lceil M/2 \rceil}$,
  we have
  \begin{equation*}
    e^{\MC}_{\prob}(n,\delta,\Wspace)
      \geq {\textstyle \frac{1}{2}} \gamma M \,.
  \end{equation*}
\end{lem}
\begin{proof}
  Let $\mu$ be the uniform distribution on the finite set
  \begin{equation*}
    \Finput
      := \left\{\sum_{i \in I} s_i f_i
                \colon
                \text{$I \subset \{1,\ldots,N\}$ with $\# I = M$,
                      $s_i \in \{\pm 1\}$}
        \right\}
      \subset \Ball_\Wspace \,.
  \end{equation*}
  Let~$Q_n \colon \Finput \to \R$ be a deterministic algorithm using
  $n$~function values.
  Without loss of generality, we may assume
  that the algorithm computes function values \mbox{$y_i = f(\vecx_i)$}
  with~\mbox{$\vecx_i \in G_i$} for~$i=1,\ldots,n$.
  Hence, from the $i$-th piece of information we learn whether~$i \in I$,
  and if so, whether~$s_i = +1$ or~$-1$ within the representation
  \begin{equation*}
    f \,=\, f_{I,\vecs}
      \,:=\, \sum_{i \in I} s_i f_i
      \,.
  \end{equation*}
  Let $m(f) := \# (I \cap \{1,\ldots,n\})$ denote the number of
  detected subdomains~$G_i$ where the function~$f$ is non-zero.
  Under~$\mu$,
  the random variable $m(\cdot)$  
  is distributed according to a hypergeometric distribution
  with population of size~$N$ containing~$M \leq N$ items of interest
  and admitting~$n<N$  
  draws without replacement.
  The expected value is
  \begin{equation*}
    \int_{\Finput} m(f) \dint\mu(f)
      \,=\, \frac{n}{N} \, M
      \,\leq\, \frac{1}{4} \, M \,,
  \end{equation*}
  and using Markov's inequality
  we conclude
  \begin{equation} \label{eq:m<=M/2}
    \mu\left\{f \colon m(f) \leq \fracts{1}{2} M \right\}
      \,\geq\, \frac{1}{2} \,.
  \end{equation}
  Given the information~$(f(\vecx_1),\ldots,f(\vecx_n)) = \vecy$ with $m(f) = m$,
  there are still \mbox{$k := M-m$} unknown~$s_i$ for subdomains~$G_i$
  where the function does not vanish
  and the conditional distribution of $\Int(f)$ 
  is given similarly to the proof of Lemma~\ref{lem:aux1}
  (only $n$ is substituted by~$m$).
  Hence, the conditional uncertainty can be quantified via a binomial sum. 
  For $0 < \eps \leq \frac{1}{2} M \gamma$, 
  up to \mbox{$k' := \lfloor \eps/\gamma \rfloor + 1\leq \lceil \frac{1}{2} M \rceil$}
  terms are removed, and we obtain
  \begin{align*} \label{eq:|Sf-Qf|>eps|y,m}
    \mu\bigl(|Q_n(f) - &\Int(f)| > \eps \;\big|\;
             (f(\vecx_1),\ldots,f(\vecx_n)) = \vecy,\,
             m(f)=m\bigr) \\
    &\geq\, 
      2^{-k} \Biggl[\sum_{j=0}^{\left\lfloor \frac{k-k'}{2}
              \right\rfloor}
              \binom{k}{j}
            + \sum_{j=\left\lceil \frac{k+k'+1}{2} \right\rceil}^k
              \binom{k}{j}
                \Biggr].
  \end{align*}
  This bound is the same for all information outcomes $\vecy$
  with the same number $m(f) = m$ of detected non-zero subdomains,
  and for $k \geq k'$, by Lemma~\ref{lem:BinomTail2}, we further estimate
  \begin{equation}\label{eq:|Sf-Qf|>eps|y,m}
    \mu\left(|Q_n(f) - \Int(f)| > \eps \mid  m(f)=m\right) \geq 2^{-k'}.
  \end{equation}
  Note that~$m \leq \frac{1}{2} \, M$
  implies~\mbox{$k = M - m
                  \geq \lceil \frac{1}{2} \, M \rceil
                  \geq k'
                  $},
  so \eqref{eq:|Sf-Qf|>eps|y,m} can be used
  under the condition formulated in \eqref{eq:m<=M/2}.
  Hence,
  \begin{align*}
    \mu\{|Q_n(f) - &\Int(f)| > \eps\} \\
      &\,\geq\, \sum_{m=0}^{\lceil M/2 \rceil}
              \mu\left(|Q_n(f) - \Int(f)| > \eps\mid m(f)=m\right)
                \cdot \mu\{ f\colon m(f)=m \} \\
      &\stackrel{\eqref{eq:|Sf-Qf|>eps|y,m}}{\,\geq\,}
        2^{-k'} \cdot \mu\{f\colon m(f) \leq \fracts{1}{2} M\}
      \,\geq\, \frac{1}{2} \, 2^{-\lceil M / 2 \rceil}.
  \end{align*}
  By Bakhvalov's trick~\eqref{eq:Bakh}
  this is a lower bound for the worst case uncertainty
  \mbox{$\sup_{\|f\|_{\Wspace} \leq 1} \P\{|A_n(f) - \Int(f)| > \eps\}$},
  and for $0 < \delta < \frac{1}{2} \, 2^{-\lceil M / 2 \rceil}$
  the implication~\eqref{eq:delta<->eps} proves the assertion.
\end{proof}

\subsection{Lower bounds for Sobolev classes}

The norms of classical Sobolev spaces $W_p^r(G)$
with \emph{integrability index~$p$}
possess the property that for any decomposition
of the support of a function~$f \in W_p^r(G)$
into essentially disjoint sub-domains, say $G_1,\dots,G_M$,
we have
\begin{equation} \label{eq:integrability}
  \left\|f\right\|_{W_p^r(G)}
    = \left(\sum_{i=1}^M \left\|f\right\|_{W_p^r(G_i)}^p\right)^{1/p}
    \leq M^{1/p} \, \max_{i=1,\ldots,M} \|f\|_{W_p^r(G_i)} \,.
\end{equation}
The \emph{smoothness} has an effect on scalings,
namely for functions $\varphi \colon \R^d \to \R$
and $\psi(\vecx) := \varphi(m \vecx - \veci)$, with $m > 0$ and $\veci \in \R^d$,
we have the following well known relation between the derivatives,
\begin{equation*}
  \|D^{\vecalpha}\psi\|_{L_p(\R^d)}
    = m^{|\vecalpha|_1 - d/p} \, \|D^{\vecalpha}\varphi\|_{L_p(\R^d)}\,,
  \quad \text{for $\vecalpha \in \N_0^d$.}
\end{equation*}
For~$m \geq 1$ this leads to the scaling property
\begin{equation} \label{eq:scaling}
  \|\psi\|_{W_p^r(\R^d)}
    \leq m^{r - d/p} \, \|\varphi\|_{W_p^r(\R^d)} \,.
\end{equation}
If $\supp \varphi, \supp \psi \subseteq G$,
then this relation holds also for the norms
of the restricted space~$W_p^r(G)$. 
\begin{thm} \label{thm:LBs}
  Let $G \subset \R^d$ be a domain with nonempty interior.
  Further, let $r \in \N_0$, $1 \leq p \leq \infty$ and define~$q := \min\{p,2\}$.
  Then we have the asymptotic lower bound
  \begin{equation*}
    e^{\MC}_{\prob}(n,\delta,W_p^r(G))
      \succeq \min\left\{ n^{-r/d},\,
                          n^{-(r/d + 1-1/q)} \, (\log\delta^{-1})^{1-1/q}
                  \right\} \,.
  \end{equation*}
\end{thm}
\begin{proof}
  Since the interior of~$G$ is nonempty there exists a cubic subdomain.
  We restrict to functions which are supported within that rectangular subdomain,
  and by scaling, without loss of generality, we may assume $G = [0,1]^d$.
  
  Let $\varphi \colon \R^d \to \R$ be a sufficiently smooth function
supported on~$[0,1]^d$
  with $\|\varphi\|_{W_p^r([0,1]^d)} \leq 1$ 
  and \mbox{$\gamma_0 := \Int \varphi > 0$}.
  We call $\varphi$ \emph{bump function}.
  For $m\in\N$ we split $[0,1]^d$ into $N = m^d$ subcubes~$G_\veci$
  with $\veci \in [m]^d := \{0,\ldots,m-1\}^d$
  and equip each subcube with a scaled, shifted bump function
  $\psi_{\veci}(\vecx) := \varphi(m\vecx - \veci)$.

  If~$2 \leq p \leq \infty$, we choose $m := \lceil(5n+6)^{1/d}\rceil$
  and $f_{\veci} := 
                        m^{-r} \, \psi_{\veci}$.
  Hence, by~\eqref{eq:integrability} and~\eqref{eq:scaling} with~$M = N$ we have
  \begin{equation} \label{eq:signedbumps}
    \Biggl\|\sum_{\veci \in [m]^d} s_{\veci} f_{\veci}\Biggr\|_{W_p^r([0,1]^d)} \leq 1
    \qquad \text{for arbitrary $s_{\veci} \in \{\pm 1\}$.}
  \end{equation}
Then
  $\gamma = \Int f_{\veci} = 
                             m^{-r-d} \, \gamma_0$.
  Restricting to $n\geq 17$ and $0 < \delta \leq 1/3$,
  we can apply Lemma~\ref{lem:aux1} and obtain
  \begin{align*}
    e^{\MC}_{\prob}(n,\delta,W_p^r([0,1]^d))
      &\geq 
          \gamma_0 \, m^{-r-d}
            \, \min\left\{n^{1/2} \sqrt{\log_4 \frac{1}{3\delta}},\,
                          n
                   \right\}\\
      &\succeq \min\{n^{-r/d-1/2} \sqrt{\log \delta^{-1}},
                     n^{-r/d}\} \,.
  \end{align*}
  
  If~$1 \leq p < 2$, we restrict to~$0 \leq \delta < 1/4$
  and choose \mbox{$m := \lceil(4n)^{1/d}\rceil$}.
  In the case $2^{-2n-1} \leq \delta$,
  we take $M = 2 \lceil \log_2(4\delta)^{-1}\rceil \leq 2 \log_2(2\delta)^{-1}$,
  and easily see that
  $M \leq 4n \leq N$ is fulfilled.
  Here, put~$f_{\veci} := 
                        m^{-r} \, (N/M)^{1/p} \, \psi_{\veci}$,
 and note that by~\eqref{eq:integrability} and \eqref{eq:scaling}
  we have
  \begin{equation*}
    \Biggl\|\sum_{\veci \in I} s_{\veci} f_{\veci}\Biggr\|_{W_p^r([0,1]^d)} \leq 1 \,,
    \qquad \text{for arbitrary $s_{\veci} \in \{\pm 1\}$\;
                 and \;$I \subseteq [m]^d$ with $\# I = M$.}
  \end{equation*}
  We have~$\gamma = \Int f_{\veci} = 
                                     m^{-r-d+d/p} \, M^{-1/p}$,
  and Lemma~\ref{lem:aux2} implies
  \begin{equation*}
    e^{\MC}_{\prob}(n,\delta,W_p^r([0,1]^d))
      \geq \fracts{1}{2} \, \gamma_0 \, 
                            m^{-r-d+d/p} \, M^{1-1/p}
      \succeq n^{-(r/d+1-1/p)} \left(\log \delta^{-1}\right)^{1-1/p} \,.
  \end{equation*}
  For small $\delta \in (0,2^{-2n-1})$, however,
  we may just choose~$M = 4n$, and similarly we obtain
  \begin{equation*}
    e^{\MC}_{\prob}(n,\delta,W_p^r([0,1]^d))
      \succeq n^{-r/d} \,,
  \end{equation*}
  which finishes the proof.
\end{proof}

\begin{rem}[Lower bounds for non-integer smoothness] \label{rem:frac-smooth}
  There are several approaches to generalize Sobolev spaces
  for non-integer smoothness $r > 0$.
  For example the \emph{Slobedeckii space} $W_p^r(G)$
  is given as the set of functions with finite norm
  \begin{equation*}
    \|f\|_{W_p^r(G)}
      := \left(\|f\|_{W_p^{\lfloor r \rfloor}(G)}^p
                 + \sum_{|\vecalpha|_1 = \lfloor r \rfloor}
                     \int_G \int_G
                       \frac{|D^{\vecalpha}f(\vecx) - D^{\vecalpha}f(\vecz)|^p
                           }{|\vecx-\vecz|^{d+(r-\lfloor r \rfloor)p}}
                       \dint\vecx \dint\vecz
         \right)^{1/p} \,,
  \end{equation*}
  where $1 \leq p < \infty$,
  see for instance the book of Triebel~\cite[p.~36]{TriI}.
  For such spaces the inequality~\eqref{eq:integrability} does not hold anymore.
  However, we can still
  construct fooling functions composed of bumps on disjoint subcubes
  with random sign, but we need to introduce an additional constant
  in order to take the non-local nature of fractional derivatives into account.
  
  Let us consider an easier example: 
  Namely, classes of H\"older continuous functions
  with fractional smoothness $0 < \beta \leq 1$
 (and integrability parameter $p=\infty$) given by
  \begin{equation} \label{eq:Cbeta}
    C^{\beta}([0,1]^d)
      := \left\{f:[0,1]^d \to \R \,\middle|\,
                |f|_{C^{\beta}}
                  := \sup_{\vecx,\vecz \in [0,1]^d} 
                       \frac{|f(\vecx)-f(\vecz)|}{|\vecx-\vecz|_{\infty}^{\beta}}
                  < \infty
         \right\} \,.
  \end{equation}
  With the choice
  $f_{\veci} := \frac{1}{2}\, m^{-\beta} \, \psi_{\veci}$,
  within the proof above, one can ensure~\eqref{eq:signedbumps}.
  Thus, loosing just a factor $1/2$, we still have the same order
  as should be expected from generalizing the integer smoothness case,
  \begin{equation} \label{eq:HoelderLB}
    e^{\MC}_{\prob}(n,\delta,C^{\beta}([0,1]^d))
      \succeq n^{-{\beta}/d} \, \min\left\{ 1,\,
                          \sqrt{\frac{\log\delta^{-1}}{n}}
                  \right\}\,.
  \end{equation}
  This fits very well to the upper bounds of Theorem~\ref{thm:HoelderStrat}.
\end{rem}

\section{Upper Bounds}
\label{sec:UBs}

\subsection{Probability amplification}

One of the most elementary methods of `probability amplification'
is the so-called `median trick', 
see Alon et al.~\cite{AMS96} and Jerrum et al.~\cite{JVV86}.
The following proposition is a minor modification of
\cite[Proposition~2.1, in particular (2.6)]{NiPo09} from Niemiro and Pokarowski,
now adapted to the language of algorithms and IBC.

As in Section~\ref{sec: aux_lem},
we consider a general function space~$\Wspace$
equipped with a (semi-)norm $\|\cdot\|_{\Wspace}$
and take its unit ball $\Ball_\Wspace$ as input set.

\begin{prop}[Median trick] \label{prop:med}
  For $\eps > 0$ 
  let $A_m$ be an arbitrary Monte Carlo algorithm such that
  \begin{equation*}
    \sup_{\|f\|_{\Wspace} \leq 1} \P\{|A_m(f) - \Int f| > \eps\} \leq \alpha \,,
  \end{equation*}
  where~$0 < \alpha < 1/2$.
  For an odd natural number~$k$,
  define
  \begin{equation*}
    A_{k,m}(f) := \median\left\{A_m^{(1)}(f),\ldots,A_m^{(k)}(f)\right\}
  \end{equation*}
  as the median of $k$ independent realisations of~$A_m$.
  Then
  \begin{equation*}
    \sup_{\|f\|_{\Wspace} \leq 1} \P\{|A_{m,k}(f) - \Int f| > \eps\}
      \,\leq\, \frac{1}{2} (4\alpha(1-\alpha))^{k/2} 
      \,<\, 2^{k-1} \, \alpha^{k/2} \,.
  \end{equation*}
\end{prop}
The previous proposition can be used to derive upper bounds
for the probabilistic $(\eps,\delta)$-complexity 
$n^{\MC}_{\prob}(\eps,\delta,\Wspace)$
in terms of the $\ell$-mean error complexity
$n^{\MC}_{\ellmean}(\eps,\Wspace)$,
compare \eqref{eq:ME}.

\begin{thm} \label{thm:med}
  Let $\ell\geq 1$ and $0 < \delta \leq 1/2$.
  Then for the $(\eps,\delta)$-complexity holds
  \begin{equation*}
    n^{\MC}_{\prob}(\eps,\delta,\Wspace)
      \leq 2 \log_2 \delta^{-1}
            \cdot n^{\MC}_{\ellmean}\left(8^{-1/\ell} \, \eps ,\, \Wspace \right) \,.
  \end{equation*}
  In particular, if $e^{\MC}_{\ellmean}(n,\Wspace) \preceq n^{-\rho}$
  for some $\rho > 0$,
  then we have
  \begin{equation*}
    e^{\MC}_{\prob}(n,\delta,\Wspace)
      \preceq \left(\frac{\log \delta^{-1}}{n}\right)^{\rho} \,
    \qquad\text{for $n \succeq \log\delta^{-1}$}.
  \end{equation*}
\end{thm}
\begin{proof}
  Without loss of generality, we assume that
  $n^{\MC}_{\ellmean}(8^{-1/\ell} \eps,\Wspace) < \infty$,
  otherwise the claimed inequality is trivial.
  Hence, there is an $m\in\N$ such that
  $e^{\MC}_{\ellmean}(m,\Wspace)<8^{-1/\ell} \eps$.
  This implies that there is a Monte Carlo
  algorithm $A_m$ such that
  \begin{equation*}
    e^{\MC}_{\ellmean}(A_m,\Wspace)
      :=\sup_{{\|f\|_{\Wspace} \leq 1}}
          \left( \expect \| A_m(f)-\Int f\|^{\ell} \right)^{1/\ell}
      \leq 8^{-1/\ell} \eps.
  \end{equation*}
  Thus, for any $f\in\Ball_\Wspace$
  by Markov's inequality we have
  \begin{equation*}
    \P\{|A_m(f) - \Int f| > \eps\}
      \leq \left(\frac{e^{\MC}_{\ellmean}(A_m,\Wspace)}{\eps}\right)^\ell
      \leq \frac{1}{8}.
  \end{equation*}
  Now we aim to apply Proposition~\ref{prop:med} with $k$ chosen as the smallest odd natural number that satisfies
  $k \geq 2 \log_2 (2\delta)^{-1}$.
  (Note that $k \leq 2 \log_2 \delta^{-1}$ for $0 < \delta \leq 1/2$.)
  Then we obtain the desired complexity bound
  \begin{equation*}
    n^{\MC}_{\prob}(\eps,\delta,\Wspace)
      \leq k \cdot n^{\MC}_{\ellmean}\bigl(8^{-1/\ell} \, \eps, \Wspace
                   \bigr)
      \leq 2 \log_2 \delta^{-1} \cdot
        \, n^{\MC}_{\ellmean}\bigl(8^{-1/\ell} \, \eps ,\, \Wspace \bigr)\,.
  \end{equation*}
  In terms of the error quantities,
  for fixed~$m$ and odd $k \geq 2 \log_2 (2\delta)^{-1}$ we can state
 \begin{equation*}
    e^{\MC}_{\prob}(km,\delta, \Wspace)
      \leq 8^{-1/\ell} \, e^{\MC}_{\ellmean}(m,\Wspace)\,.
  \end{equation*}
  Assuming $e^{\MC}_{\ellmean}(m,\Wspace) \leq C \, m^{-\rho}$
  for $m \geq m_0 \in \N$,
  and given an information budget $n \geq 2 m_0 \log_2{\delta^{-1}}$,
  put $m := \lfloor n/(2 \log_2 \delta^{-1}) \rfloor \asymp n/\log\delta^{-1}$.
  Then the assertion follows by the assumption
  with hidden constant
  $\bigl(\frac{2}{\log 2}(1+1/m_0)\bigr)^{\rho} \cdot C$.
\end{proof}

\begin{rem}[Integrating $L_p$-functions]
  The lower bounds of Theorem~\ref{thm:LBs}
  match the upper bounds from Theorem~\ref{thm:med},
  iff the rate of convergence is related to the integrability index by
  $\rho = 1-1/q$ where $q := \min\{p,2\}$.
  This is only the case for smoothness $r = 0$, 
  i.e.,\ when $L_p$-balls are the considered input sets.
 In that case,
  \begin{equation} \label{eq:INT-Lp}
    e^{\MC}_{\prob}(n,\delta,L_p)
      \asymp \left(\frac{\log \delta^{-1}}{n}\right)^{1-1/q} \,,
    \qquad\text{for $n \succeq \log \delta^{-1}$.}
  \end{equation}
  Here we used estimates of the $q$-mean error for the standard i.i.d.-based
  Monte Carlo method applied to $L_p$-functions which, for example, can be found in 
  \cite[Theorem~2]{vBEs65}, \cite[Proposition~5.4]{He94},
  \cite[Sect.~2.2.8, Proposition~3]{No88},
  as well as \cite[Proof of Theorem~1]{RuSc15}.
\end{rem}

\begin{rem}[Alternatives to the median trick]
  Catoni~\cite{Ca12} proposes an alternative scheme based on random samples
  of $L_p$-functions, which suppresses outliers
  and (in contrast to the median trick) is symmetric
  in the sense that permuting the sample data does not change the result.
  Compare also Huber~\cite{Hu17} where this approach is combined with the median trick.
  Unfortunately, their methods
  are not homogeneous and shift invariant, that is, for the algorithm, say $A$,
  in general
  $A(af + b) = a A(f) + b$ does not hold.
\end{rem}

\subsection{Separation of the main part}

\emph{Separation of the main part}, also known as \emph{control variates},
is a well established technique of variance reduction which uses
the approximation of functions 
with respect to an~$L_q$-norm
in order to exploit the smoothness of the given input set.

Within this section we assume that $G\subset \R^d$ is a bounded Lipschitz domain,
see~\cite{NT06} for details.
For $q\geq 1$ let $L_q(G)$ be
the Lebesgue space equipped with the norm $\| \cdot \|_{L_q(G)}$. 
Let $\Wspace\subset L_q(G)$ be a normed linear space
with corresponding unit ball $\Ball_\Wspace$
and assume that function evaluations are continuous (well-defined) on $\Wspace$.
For the approximation step we only consider linear methods 
\begin{equation} \label{eq:linAppMethod}
  A_n \colon \Wspace \to L_q(G), \qquad
  f \mapsto g := \sum_{i=1}^n f(\vecx_i) \, g_i \,,
\end{equation}
with nodes~$\vecx_i \in G$ and functions~$g_i \in L_q(G)$, where $\int_G g_i(\vecx) \dint \vecx$ is known
for any $i=1,\dots,n$.
The minimal $L_q$-approximation error of such methods 
is denoted by 
\begin{equation} \label{eq:a(n)}
  e^{\deter}(n,\Wspace \embed L_q)
    := \inf_{A_n} \sup_{\|f\|_{\Wspace} \leq 1} \|A_n(f) - f\|_{L_q(G)} \,,
\end{equation}
and the $\eps$-complexity $n^{\deter}(\eps,\Wspace \embed L_q)$ is
the minimal number of function evaluations needed in order to achieve 
an $L_q$-approximation error smaller than $\eps$.
The idea is to apply a Monte Carlo integration method $M_n \colon L_q \to \R$
to the difference $f-g$ between approximating and original function,
while the integral of~$g = A_n(f)$ is considered to be known.
This approach leads to the following theorem.

\begin{thm}[Separation of the main part] \label{thm:separation}
For any $n\in\N$ we have
\begin{equation*}
    e^{\MC}_{\prob}(2n,\delta,\Wspace)
      \leq e^{\deter}(n,\Wspace \embed L_q) 
            \cdot e^{\MC}_{\prob}(n,\delta,L_q) \,. 
  \end{equation*}
\end{thm}
\begin{proof}
  Let $A_n$ be a linear approximation method, see~\eqref{eq:linAppMethod},
  which guarantees for any $f\in \Ball_\Wspace$ and some $\alpha > 0$ that
  \begin{equation*}
   \|A_n(f) - f\|_{L_q(G)} \leq \alpha.
  \end{equation*}
  Further, let $M_n^{\omega}$ be a Monte Carlo method which 
  approximates $\Int h$ for inputs $h \in L_q$. 
  With this we define a new Monte Carlo method $Q_{2n}^\omega$
  (a randomized quadrature rule) as follows:
\begin{enumerate}
    \item Compute the approximation $g := A_n(f) \in L_q(G)$,
      using function values~$f(\vecx_i)$
      at nodes $\vecx_1,\ldots,\vecx_n \in G$.
    \item Return
      \begin{equation*}
        Q_{2n}^{\omega}(f)
        := \Int g + \alpha M_n^{\omega}\left(\alpha^{-1}(f - g)\right)
        \end{equation*}
      where $M_n$ evaluates 
      $\alpha^{-1}(f - g)$ at random nodes
      \mbox{$\vecX_{n+1}^{\omega},\ldots,\vecX_{2n}^{\omega} \in G$.}
      (If $M_n$ is homogeneous, that is,
      \mbox{$M_n^{\omega}(\lambda f) = \lambda M_n^{\omega}(f)$}
      for any $\lambda \in \R$,
      then $\alpha$ cancels out.)
  \end{enumerate}
  Note that, the information
  \mbox{$\vecy = (f(\vecx_1),\ldots,f(\vecx_n),
    f(\vecX_{n+1}^{\omega}),\ldots,f(\vecX_{2n}^{\omega}))$}
  suffices to execute the algorithm, namely,
  \begin{align*}
    \Int g
      &= \sum_{i=1}^n f(\vecx_i) \Int g_i \,,
    \qquad \text{and} \\
    [\alpha^{-1}(f - g)](\vecX_{n+j}^{\omega})
      &= \alpha^{-1} \left(f(\vecX_{n+j}^{\omega})
                        - \sum_{i=1}^n f(\vecx_i) \, g_i(\vecX_{n+j}^{\omega})
                  \right)
    \qquad \text{for $j=1,\ldots,n$.}
  \end{align*}
  Indeed, $\Int g_i$ is assumed to be precomputed and
  computing function values of $g_i$
  is considered to belong to the combinatorial cost,
  rather than the information cost of the algorithm.
  
  By writing $\eps = \alpha \eps'$,
  the uncertainty of the algorithm can be traced back to the uncertainty of~$M_n$.
  Namely, if $M_n$ is $(\eps',\delta)$-approximating in $L_q(G)$, then
  \begin{align*}
    \P\left\{|Q_{2n}^{}(f) - \Int f| > \eps\right\}
      &= \P\left\{|M_{n}^{}(\alpha^{-1}(f-g)) - \Int (\alpha^{-1}(f-g))|
                     > \eps'
           \right\} \\
      &\leq \delta \,.
  \end{align*}
  Optimal methods $A_n$ and $M_n^{\omega}$ lead to
  $\alpha \to e^{\det}(n,\Wspace \embed L_q)$
  and $\eps' \to e^{\MC}_{\prob}(n,\delta,L_q)$,
  while keeping the uncertainty bounded by~$\delta$,
  thus letting $\eps$ approach the stated error bound.
\end{proof}

As long as function evaluations are continuous,
it suffices to work with deterministic approximation methods of the form $\eqref{eq:linAppMethod}$. 
Note that for isotropic Sobolev spaces $W_p^r(G)$
on bounded Lipschitz domains~$G \subset \R^d$,
this is the case iff $r p > d$.
In this setting 
it is well known that with $q := \min\{p,2\}$,
\begin{equation} \label{eq:WpLq}
  e^{\deter}\left(n,W_p^r(G) \embed L_q(G)\right)
    \asymp n^{-r/d} \,,
  \qquad\text{if $rp > d$.}
\end{equation}
For $G = [0,1]^d$,
this result can be achieved with piecewise polynomial interpolation,
see for instance Heinrich~\cite[Proposition~5.1]{He94},
technical details of approximation methods are contained in Ciarlet~\cite{Cia78}.
For the general case of bounded Lipschitz domains~\mbox{$G \subset \R^d$},
see Novak and Triebel~\cite[Theorem~23]{NT06}.
From this we conclude optimal upper bounds.

\begin{thm} \label{thm:UBsWprSep}
  Let $G \subset \R^d$ be a bounded Lipschitz domain.
  Further let $r \in \N$ and $1 \leq p \leq \infty$ with $rp > d$.
  Then we have the asymptotic rate
  \begin{equation*}
    e^{\MC}_{\prob}\left(n,\delta,W_p^{r}(G)\right)
      \asymp n^{-r/d} \, \min\left\{1, \,
                                    \left(\frac{\log \delta^{-1}}{n}\right)^{1-1/q}
                             \right\} \,,
  \end{equation*}
  where~$q := \min\{2,p\}$.
\end{thm}
\begin{proof}
  The lower bounds follow from Theorem~\ref{thm:LBs}.
  
  For $n \succeq \log \delta^{-1}$,
  we combine~\eqref{eq:INT-Lp} with~\eqref{eq:WpLq}
  via Theorem~\ref{thm:separation} and obtain
  \begin{equation*}
    e^{\MC}_{\prob}\left(n,\delta,W_p^{r}(G)\right)
      \preceq n^{-r/d} \, \left(\frac{\log \delta^{-1}}{n}\right)^{1-1/q} \,.
  \end{equation*}
  For $n \prec \log \delta^{-1}$ we rely on deterministic quadrature.
  This problem is easier than approximation in the sense that
  if $g \in L_q(G)$ is an approximation of~$f$,
  then \mbox{$|\Int f - \Int g| \leq \Vol_d(G)^{1-1/q} \cdot \|f-g\|_{L_q(G)}$}.
  Hence,
  \begin{align*}
    e^{\deter}(n,W_p^{r}(G))
      &:= \inf_{A_n} \sup_{\|f\|_{\Wspace} \leq 1} |A_n(f) - \Int f| \\
      &\preceq e^{\deter}(n,W_p^{r}(G) \embed L_q(G))
       \asymp n^{-r/d} \,.
  \end{align*}
  See also Novak~\cite[1.3.12]{No88} for a direct derivation on $G = [0,1]^d$.
\end{proof}

\begin{rem}[Lower smoothness]
  The condition $rp>d$ is necessary to guarantee that the evaluation
  of functions on $W_p^r(G)$ for $1<p\leq \infty$ is well-defined.
  (For $p = 1$, the condition $r = d$ is also sufficient, but then deterministic
  methods already provide the optimal error rates.)
  In the cases $1 < p < \infty$ with $rp \leq d$ one can still use
  separation of the main part, but with a randomized approximation scheme,
  see Heinrich~\cite{He08SobI} for the case~$G = [0,1]^d$.
  That way, for any $1 \leq p \leq \infty$
  and general~$r \in \N$ we have
  \begin{equation*}
    e^{\MC}_{\oneMean}(n,W_p^r([0,1]^d)) \asymp n^{-(r/d + 1 - 1/q)}\,,
  \end{equation*}
  with~$q := \min\{p,2\}$.
  Probability amplification, see Theorem~\ref{thm:med},
  yields
  \begin{equation*}
    e^{\MC}_{\prob}(n,\delta,W_p^r([0,1]^d))
      \asymp \left(\frac{\log \delta^{-1}}{n}\right)^{r/d + 1 - 1/q}\,,
    \qquad\text{for $n \succeq \log \delta^{-1}$.}
  \end{equation*}
  The power of $\log \delta^{-1}$ in this upper bound
  may exceed the power of the lower bound by $r/d$ which can get close to~$1$
  for 
  $p \to 1$.
  We conjecture that the 
  influence of $\delta$ 
  is smaller
  at least if one is close to the regime where functions are continuous.
  
  Instead of deterministic algorithms of the form \eqref{eq:linAppMethod}
  one might also consider general randomized methods for the approximation of functions in $\Wspace$.
  For those let
  $
   e^{\MC}_{\prob}(n,\delta,\Wspace\embed L_q)
  $
  be the smallest $\eps>0$, such that there exists a general randomized approximation algorithm
  satisfying
  \begin{equation*}
    \P\{\| A_n(f)-f\|_{L_q} > \eps \} \leq \delta
    \qquad \text{for all $\|f\|_{\Wspace} \leq 1$.}
  \end{equation*}
  Similarly to Theorem~\ref{thm:separation} one can show that
  \begin{equation}
    e^{\MC}_{\prob}(2n,\delta,\Wspace)
      \leq e^{\MC}_{\prob}(n,\delta/2,\Wspace \embed L_q)
            \cdot e^{\MC}_{\prob}(n,\delta/2,L_q) \,.
  \end{equation}
  Such an approach, however, seems to rely on complicated algorithms,
  since non-linearity might be inevitable
  in order to suppress outliers.
  There might be easier implementable Monte Carlo methods for integration
  which achieve a better order of convergence without relying on the approximation of functions.
  Such methods are needed in spaces of mixed smoothness,
  see the discussion in Section~\ref{sec:Wmix}.
  
  Anyway, studying Sobolev embeddings $W_p^r(G) \embed L_q(G)$
  in terms of approximation with high confidence
  within the regime \mbox{$d(q - p) < rqp \leq dq$}
  for general integrability parameters \mbox{$1 \leq p,q < \infty$},
  is an interesting problem on its own,
  compare Heinrich~\cite{He08SobI}.
\end{rem}

\subsection{Stratified sampling}

Let us introduce stratified sampling for the approximation of $\Int f$
for integrable functions defined on~$G = [0,1]^d$. 
For $m\in\N$ we split the unit cube $[0,1]^d$ into $m^d$~subcubes given by
\begin{equation*}
 G_{\veci} = \prod_{j=1}^d \left[ \frac{i_j}{m},\frac{i_j+1}{m}\right),
\end{equation*}
with  $\veci\in [m]^d := \{0,\dots,m-1\}^d$ and $\veci = (i_1,\dots,i_d)$.
Let $(\vecX_{\veci})_{\veci\in [m]^d}$ be a sequence of independent random variables
with $\vecX_{\veci}$ uniformly distributed in $G_{\veci}$.
Then, stratified sampling is given by
\begin{equation} \label{eq:strat}
  S_m^d(f) := \frac{1}{m^d} \sum_{\veci \in [m]^d} f(\vecX_{\veci}) \,,
\end{equation}
which uses $m^d$ function evaluations of $f$.
Compared to the separation of the main part,
stratified sampling is easier to implement.
In some cases we show that 
it provides
optimal results in terms of the $(\eps,\delta)$-complexity.
Since the structure only depends on the information budget
and not on~$\delta$ (compare the median trick),
we obtain a universal method in terms of the uncertainty.

We use Hoeffding's inequality which is, for the convenience of
the reader, stated in the following proposition.

\begin{prop}[Hoeffding's inequality] \label{prop:Hoeffding}
  Let~$Y_1,\ldots,Y_n$ be independent random variables
  supported on intervals of length~$b_i > 0$, that is,
  $\esssup Y_i - \essinf Y_i \leq b_i$.
  Then, for $S_n := \frac{1}{n} \sum_{i=1}^n Y_i$
  and~$\eps > 0$ we have
  \begin{equation*}
    \P\left\{|S_n - \expect S_n| > \eps \right\}
      \leq 2 \exp\left(-\frac{2 \, n^2 \, \eps^2}{\sum_{i=1}^n b_i^2}\right) \,.
  \end{equation*}
\end{prop}

First, we consider
H\"older classes $C^{\beta}([0,1]^d)$ with smoothness~$\beta \in (0,1]$,
see~\eqref{eq:Cbeta}.
Compare also~\cite{Bakh59} for the result in terms of the root mean squared error.

\begin{thm} \label{thm:HoelderStrat}
  For the classes of H\"older-continuous functions on $[0,1]^d$,
  stratified sampling achieves the optimal rate of convergence, namely
  \begin{equation*}
    e^{\MC}_{\prob}\left(n,\delta,C^{\beta}([0,1]^d)\right)
      \asymp n^{-\beta/d}
        \, \min\left\{1,\,\sqrt{\frac{\log \delta^{-1}}{n}}\right\} \,.
  \end{equation*}
\end{thm}
\begin{proof}
  Concerning the lower bounds,
  see \eqref{eq:HoelderLB} in Remark~\ref{rem:frac-smooth}.
  
  For the upper bounds,
  we start with the case~$n = m^d$ with $m \in \N$
  and employ~$S_m^d$. Obviously, this method is unbiased,
  i.e.,\ $\expect S_m^d(f) = \Int f$.
  By H\"older continuity, the random variables
  $Y_{\veci} = f(\vecX_{\veci})$ are spread on intervals
  of length $b_{\veci} \leq m^{-\beta}$ for $|f|_{C^{\beta}} \leq 1$.
  This implies \mbox{$|S_m^d(f) - \Int f| \leq m^{-\beta} = n^{-\beta/d}$}.
  Hoeff\-ding's inequality, Proposition~\ref{prop:Hoeffding},
  leads to
  \begin{equation*}
    \P\{|S_m^d(f) - \Int f| > \eps\}
      \leq 2 \exp\left( - \frac{2 \, m^{2d} \, \eps^2
                              }{m^d \cdot m^{-2\beta}}
                 \right)
      = 2 \exp\left(- 2 \, m^{d+2\beta} \eps^2\right) \,.
  \end{equation*}
  This is guaranteed to be at most~$\delta$ for
  \begin{equation*}
    \eps = \frac{1}{\sqrt{2}}
             \, m^{-(\beta + d/2)}
             \, \sqrt{\log \frac{2}{\delta}}
         = \frac{1}{\sqrt{2}}
             \, n^{-(\beta/d + 1/2)}
             \, \sqrt{\log \frac{2}{\delta}} \,.
  \end{equation*}
  
  Given an information budget $n \in \N$,
  we choose~$m := \lfloor n^{1/d} \rfloor$.
  Employing the method $S_m^d$, see~\eqref{eq:strat},
  we actually only use $m^d \leq n$ function values.
  For $n \geq 2^d$ we have $m \geq \frac{1}{2}\,n^{1/d}$,
  hence we obtain the stated asymptotics.
\end{proof}

Now we consider the isotropic Sobolev classes $W_p^1([0,1]^d)$ of smoothness~$1$.
For Hoeffding's inequality to be applicable,
we need $W_p^1([0,1]^d) \embed L_{\infty}([0,1]^d)$
which is the case for~$p > d$.

\begin{thm}
  Stratified sampling
  leads to
  \begin{equation*}
    e^{\MC}_{\prob}(n,\delta,W_p^1([0,1]^d))
      \preceq n^{-1/d}
          \, \min\left\{1,\, n^{-(1-1/q)} \sqrt{\log \delta^{-1}}\right\}\,,
      \qquad\text{for $p > d$,}
  \end{equation*}
  where $q := \min\{p,2\}$.
  For~$p \geq 2$,
  this perfectly matches 
  the lower bounds from Theorem~\ref{thm:LBs}.
\end{thm}
\begin{proof}
  We start with the one-dimensional case considering
  the method~$S_n^1$, see~\eqref{eq:strat}.
  Hence, the unit interval~$[0,1]$ is split into intervals $G_0,\ldots,G_{n-1}$
  of length~$n^{-1}$.
  For $x_1 < x_2$ from $[0,1]$ we have
  \begin{equation*}
    |f(x_2) - f(x_1)|
      = \left|\int_{[x_1,x_2]} f'(x) \dint x\right|
      \leq \int_{[x_1,x_2]} |f'(x)| \dint x \,.
  \end{equation*}
  Hence, on the $i$th interval $G_i$ we have
  \begin{equation*}
    b_i
      \leq \int_{G_i} |f'(x)| \dint x
      \leq n^{-1} \left(n \int_{G_i} |f'(x)|^q \dint x\right)^{1/q}
      = n^{-(1-1/q)} \, \|f'\|_{L_q(G_i)} \,,
  \end{equation*}
  where we used Jensen's inequality.
  Furthermore
  \begin{equation*}
    \|f'\|_{L_q([0,1])}
      = \left(\sum_{i=1}^n \|f'\|_{L_q(G_i)}^q\right)^{1/q}
      \geq n^{1-1/q} \, \left(\sum_{i=1}^n b_i^q\right)^{1/q}
\geq
        n^{1-1/q} \, \left(\sum_{i=1}^n b_i^2\right)^{1/2} \,,
  \end{equation*}
  exploiting $q\leq 2$ in the last inequality.
  Applying Hoeffding's inequality, Proposition~\ref{prop:Hoeffding},
  for $\|f\|_{W_p^1([0,1])} \leq 1$ we obtain
  \begin{equation*}
    \P\{|S_n^1(f) - \Int f| > \eps\}
        \leq 2 \exp\left( - 2 \, n^{4-2/q} \, \eps^2\right) \,.
  \end{equation*}
  This is guaranteed to be at most~$\delta$ for
  \begin{equation*}
    \eps = \frac{1}{\sqrt{2}} \, n^{-(2 - 1/q)}
             \, \sqrt{\log \frac{2}{\delta}} \,,
  \end{equation*}
  which shows the assertion for $d=1$.
  
  In higher dimension, $d \geq 2$,
  splitting~$[0,1]^d$ into~$m^d$ subcubes~$G_{\veci}$ with \mbox{$\veci \in [m]^d$},
  we exploit the embedding $W_p^1([0,1]^d) \embed L_{\infty}([0,1]^d)$.
  Namely, incorporating scaling 
  we bound the spread of function values within~$G_{\veci}$ by
  \begin{equation*}
    b_{\veci}
      := \esssup_{G_{\veci}} f - \essinf_{G_{\veci}} f
      \leq C \, m^{d/p-1} \, \|f\|_{W_p^1(G_{\veci})}
      \,,
  \end{equation*}
  with some constant~$C>0$ depending only on~$p$ and $d$.
  From this, with~$p > d \geq 2$, we conclude
  \begin{equation*}
    \Biggl(\sum_{\veci \in [m]^d} b_{\veci}^2\Biggr)^{1/2}
      \leq m^{d(1/2-1/p)}
        \Biggl(\sum_{\veci \in [m]^d} b_{\veci}^p\Biggr)^{1/p}
      \leq C \, m^{-(1-d/2)} \, \|f\|_{W_p^1([0,1]^d)} \,,
  \end{equation*}
  compare~\eqref{eq:integrability}.
  Choosing $m := \lceil n^{1/d} \rceil$, we obtain the right order by
  applying Hoeffding's inequality similarly to the one-dimensional case.
\end{proof}

\begin{rem}
  The one-dimensional problem
  contains cases of small integrability~$1 < p < 2$
  for which we do not obtain the optimal $\delta$-dependence.
  It is not known to us whether this is
  a deficiency of the method or of the proof.
  In that case, we may use separation of the main part,
  which is equally simple as~$f$ may be approximated on~$G_{\veci}$
  by just one function value.
  
  In the case of discontinuous functions, $p < d$,
  it remains challenging to find methods which detect and discourage outliers
  within stratified sampling.
  One idea might be to take several function values out of each subcube.
  This could improve also on the above mentioned case~$d=1$ and $1 < p < 2$.
  Any result in that direction might offer reasonable alternatives
  to control variates,
  where the case of small smoothness is also open.
\end{rem}

\section{Challenges in mixed smoothness spaces}
\label{sec:Wmix}

In the recent years spaces of dominating mixed smoothness
gained a lot of interest in the study of high-dimensional problems.
For a survey on this topic we refer to the paper of
D\~ung, Temlyakov, and Ullrich~\cite{DTU18*}.

For integer smoothness~$r \in \N$
and integrability $1 \leq p \leq \infty$,
on domains $G \subset \R^d$,
Sobolev spaces of dominating mixed smoothness can be defined by
\begin{equation*}
  \Wmix_p^{\mix,r}(G)
    :=\biggl\{f \in L_p(G) \,\bigg|\,
              \|f\|_{\Wmix_p^{{\mix},\,r}(G)}
                := \biggl(\sum_{\substack{\vecalpha \in \N_0^d\\
                                   |\vecalpha|_{\infty} \leq r}}
                     \|D^{\vecalpha} f\|_{L_p(G)}^p\biggr)^{1/p}
                \leq \infty
       \biggr\} \,.
\end{equation*}

Lower bounds 
of the integration problem
can be shown
by scaling bump functions $\varphi \colon [0,1]^d \to \R$ in one coordinate, that is,
for~$m \in \N$ we define functions
\mbox{$\psi_i(\vecx) := \varphi(m x_1 - i,x_2,\ldots,x_d)$},
where $i \in \{0,\ldots,m-1\} = [m]$.
By using those, similarly to Theorem~\ref{thm:LBs}
one can obtain
\begin{equation} \label{eq:WmixLB}
  e^{\MC}_{\prob}(n,\delta,\Wmix_p^{\mix,r}([0,1]^d))
    \succeq \min\left\{ n^{-r},\,
                        n^{-(r + 1-1/q)} \, (\log\delta^{-1})^{1-1/q}
                \right\} \,.
\end{equation}

When talking about upper bounds it is useful to note that the integration problem
is as difficult for the non-periodic spaces~$\Wmix_p^{\mix,r}(G)$
as for the zero-boundary space 
$\mathring\Wmix_p^{\mix,r}(G)
  := \{f \in \Wmix_p^{\mix,r}(\R^d) \mid \supp f \subseteq G\}$.
Namely, the integral of any function~$f \in \Wmix_p^{\mix,r}([0,1]^d)$,
via a change of variables, can be traced back to the integral of a function
$h := |\det\Phi'| \cdot (f \circ \Phi) \in \mathring{\Wmix}_p^{\mix,r}([0,1]^d)$
with zero boundary condition,
where $\Phi\colon [0,1]^d \to [0,1]^d$ is a smooth bijection.
That way we only loose a constant,
see Nguyen, Ullrich, and Ullrich~\cite{NgUU17}.
Let us mention that our lower bounds are based on bump functions with zero boundary,
so the lower bounds hold with the same constants.

The optimal order of convergence in terms of the root mean squared error
is determined by Ullrich~\cite{mU17}, namely
\begin{equation*}
  e^{\MC}_{\twoMean}(n,\Wmix_p^{\mix,r}([0,1]^d))
    \asymp n^{-(r+1-1/q)} \,, \qquad\text{for $r \geq \max\{1/p-1/2,0\},$}
\end{equation*}
where $q = \min\{p,2\}$.
The result is based on 
a randomly shifted and dilated Frolov rule,
developed by Krieg and Novak~\cite{KrN17}, given by
\begin{equation*}
  Q_{B,\vecv}(f)
    := \frac{1}{|\det B|}
          \sum_{\vecm \in \Z^d}
            f(B^{-\top}(\vecm+\vecv))
    \,,
\end{equation*}
where $f \in \mathring{\bf W}_p^{\mix,r}([0,1]^d)$,
which is of course only evaluated inside~$[0,1]^d$.
Here, \mbox{$B = \diag(\vecu) B_n$} with dilation random variable $\vecu$ and
independent shift random variable $\vecv$ distributed according to the uniform
distribution in $[1/2,3/2]^d$ and $[0,1]^d$, respectively, as well as 
a suitable generator matrix \mbox{$B_n = n^{1/d} B_1$}.
`Suitable' means~$\det B_n = n$ and
$\prod_{j=1}^d |(B_1 \vecm)_j| \geq c > 0$
for all $\vecm \in \Z^d\setminus\{0\}$.
In particular, the expected number of function evaluations is~$n$.
(As mentioned before, the lower bounds from Section~\ref{sec:LBs}
can be extended to methods with varying cardinality
which will only affect constants.)
Via Theorem~\ref{thm:med} one can build a method by independent repetition
which provides
\begin{equation}
  e^{\MC}_{\prob}(n,\delta,\Wmix_p^{\mix,r}([0,1]^d))
    \preceq \left(\frac{\log \delta^{-1}}{n}\right)^{r + 1-1/q}
    \qquad \text{for $n \succeq \log \delta^{-1}$,}
\end{equation}
with~$q := \min\{p,2\}$.
Unfortunately, we do not achieve the optimal dependence on~$\delta$.
The original algorithm alone does not possess desirable confidence guarantees,
as the following one-dimensional counter example shows.
This is not surprising as the number of random parameters
is fixed by the dimension and
thus we do not expect to observe concentration phenomena,
which is in contrast to stratified sampling.

\begin{example}
  We consider the integration problem in a one dimensional setting on $\mathring{\Wmix}_2^{\mix,r}([0,1])$.
  The random Frolov rule that uses $n$~function values on average
  is determined by
  \begin{equation*}
    Q_n(f) := \frac{1}{un} \sum_{m \in \Z} f\left(\frac{m+v}{un}\right) \,,
  \end{equation*}
  with independent random variables $u$ and $v$ uniformly distributed in $[1/2,3/2]$ and $[0,1]$,
  respectively.
  Let~$\varphi \in \mathring{\Wmix}_2^{\mix,r}([0,1])$ be a bump function
  with integral $\gamma_0 := \int_0^1 \varphi \dint x$
  and norm~$\|\varphi\|_{\Wmix_2^{\mix,r}} \leq 1$.
  For 
  \begin{equation*}
    f_n(x) := (2n)^{-r} \sum_{k=0}^{n-1} \varphi(2nx-2k) \,
  \end{equation*}
  observe that $\|f_n\|_{\Wmix_2^{\mix,r}} \leq 1$ and
  $\int_0^1 f_n \dint x = \gamma_0/2^{r+1} \cdot n^{-r}$.
  Furthermore, the algorithm returns~$0$
  if all the function values are computed inside
  $\bigcup_{k=0}^{n-1} \left[\frac{2k+1}{2n},\frac{k+1}{n}\right]$,
  which is where~$f_n$ vanishes.
  This happens
  if \mbox{$\frac{v}{un} \in [\frac{1}{2n},\frac{1}{n}]$}
  and \mbox{$\frac{n-1+v}{un} \in [\frac{2n-1}{2n},1]$},
  in particular for shifts \mbox{$v \in [\frac{1}{2},\frac{3}{4}]$}
  and dilations \mbox{$u \in [1-\frac{1}{4n},1]$}.
  This means, with probability exceeding $\delta_n := \frac{1}{16n}$
  the error is $\gamma_0/2^{r+1} \cdot n^{-r}$,
  hence,
  \begin{equation*}
    e^{\MC}_{\prob}(Q_n,\delta_n,\mathring{\Wmix}_2^{\mix,r}([0,1]))
      \succeq n^{-r}
      \succ n^{-(r+1/2)} \, \sqrt{\log \delta_n^{-1}}
      \asymp n^{-(r+1/2)} \, \sqrt{\log n} \,.
  \end{equation*}
  This reveals a significant gap to the general lower bound~\eqref{eq:WmixLB}.
  \qed
\end{example}

Separation of the main part does not provide the optimal rate in~$n$,
but the dependence on~$\delta$ can be reduced.
Since we may restrict to the integration problem
for functions with zero boundary condition, $\mathring{\Wmix}_p^{\mix,r}([0,1])$,
we may apply results for the approximation of periodic functions, denoted by
$\widetilde{\Wmix}_p^{\mix,r}([0,1]^d)$. Namely,
\begin{equation*}
  e^{\deter}(n,\widetilde{\Wmix}_p^{\mix,r}([0,1]^d) \embed L_p)
    \preceq n^{-r} \, (\log n)^{(r+1/2)(d-1)} \,,
  \qquad\text{for~$1 < p < \infty$,}
\end{equation*}
which can be found in~\cite[(5.11)]{DTU18*}.
Applying Theorem~\ref{thm:separation},
for \mbox{$n \succeq \log \delta^{-1}$}
we conclude that
\begin{equation}
  e^{\MC}_{\prob}(n,\delta,\Wmix_p^{\mix,r}([0,1]^d))
    \preceq n^{-(r+1-1/q)}
              \, (\log n)^{(r+1/2)(d-1)}
              \, (\log \delta^{-1})^{1-1/q} \,,
\end{equation}
where~$q := \min\{p,2\}$.
Here the $\delta$-dependence is optimal,
but the rate in $n$ is affected by logarithmic terms.

Finally, deterministic quadrature is known to achieve 
\begin{equation} \label{eq:WmixDeter}
  e^{\deter}(n,\Wmix_p^{\mix,r}([0,1]^d))
    \asymp n^{-r} (\log n)^{(d-1)/2} \,,
    \qquad\text{for~$1 < p < \infty$,}
\end{equation}
see~\cite[Theorem~8.14]{DTU18*}.
This catches the case~$n \prec \log \delta^{-1}$.

It remains a challenging open problem to find randomized integration methods which have
the right dependence on the uncertainty
while fully exploiting the smoothness.

\appendix

\section{Technical Proofs}

In Section~\ref{sec:LBs} we need the following two inequalities
about binomial sums. 
The first lemma is a minor extension of \cite[Proposition~7.3.2]{MaVo},
holding also for odd~$k$, and with slightly improved constants.

\begin{lem} \label{lem:BinomTail1}
  For~$k \in \N$ and $t \in \N_0$ we have
  \begin{multline*}
    2^{-k} \sum_{j=0}^{\lfloor k/2 \rfloor - t} \binom{k}{j}
      \,=\, 2^{-k} \sum_{j=\lceil k/2 \rceil+t}^{k} \binom{k}{j} \\
      \,\geq\, \frac{1}{2 + 4/\sqrt{\pi}}
                  \,\begin{cases} \displaystyle
                      \exp\left(-\frac{16 \, (\log 2) \, t^2}{k
                                                                  }
                          \right)
                        \quad&\text{for odd~$k$ and $t \in [0,\frac{k+3}{8}]$,}
                      \vspace{5pt} \\
                      \displaystyle
                      \exp\left(-\frac{16 \, (\log 2) \, (t-1/2)^2}{k
                                                                        }
                          \right)
                        \quad&\text{for even~$k$ and $t \in [0,\frac{k+6}{8}]$.}
                    \end{cases}
  \end{multline*}
\end{lem}
\begin{proof}
  First, recall that~\mbox{$\binom{k}{\lfloor k/2 \rfloor}
    < 2^k / \sqrt{\pi \, \lceil k/2 \rceil}$},
  which for even $k$ follows from Stirling's formula and
  for odd~$k$ can be derived from~$k+1$ via Pascal's rule.
  Hence,
  \begin{align*}
    2^{-k} \sum_{j=0}^{\lfloor k/2 \rfloor - t} \binom{k}{j}
      &\,\geq\, \frac{1}{2} - 2^{-k} \, t \, \binom{k}{\lfloor k/2 \rfloor}
      \,>\, \frac{1}{2} - \frac{t}{\sqrt{\pi \, \lceil k/2 \rceil}} \,. 
  \end{align*}
  For~\mbox{$0 \leq t \leq \sqrt{\lceil k/2 \rceil}/(1+2/\sqrt{\pi})$},
   this gives the absolute lower bound $\frac{1}{2 + 4/\sqrt{\pi}}$.
  
  For larger~$t$ we follow the approach of \cite[Proposition~7.3.2]{MaVo}.
  Basic estimates yield
  \begin{align*}
    2^{-k} \sum_{j=0}^{\lfloor k/2 \rfloor - t} \binom{k}{j}
      &\,\geq\, 2^{-k}
            \sum_{j=\lfloor k/2 \rfloor - 2t + 1}^{\lfloor k/2 \rfloor - t}
            \binom{k}{j} \\
      &\,\geq\, 2^{-k} \, t \, \binom{k}{\lfloor k/2 \rfloor-2t+1} \\
      &\,=\, 2^{-k} \, t \, \binom{k}{\lfloor k/2 \rfloor} 
            \prod_{i=1}^{2t-1}
            \frac{\lfloor k/2 \rfloor - 2t + 1 + i}{\lceil k/2 \rceil + i}\\
        &\,\geq\, 2^{-k} \, t \, \binom{k}{\lfloor k/2 \rfloor} 
    \left(\frac{\lfloor k/2 \rfloor - 2t + 2}{\lceil k/2 \rceil + 1} \right)^{2t-1}
            \,.
  \end{align*}
  Next, we use \mbox{$1-x \geq \exp(-2 \,(\log 2)\, x)$}
  for~\mbox{$0 \leq x \leq 1/2$}.
  For odd~$k$ we set $x = 2t/(\lceil k/2 \rceil + 1)$ and 
  for even~$k$ we set~\mbox{$x = (2t - 1)/(k/2 + 1)$}.
  This is where~$t \leq (k+6)/8$ for even~$k$,
  and~$t \leq (k+3)/8$ for odd~$k$,
  comes into play.
  Finally, we use
  \mbox{$\binom{k}{\lceil k/2 \rceil}
           \geq 2^k/(2 \sqrt{\lceil k/2 \rceil})$},
  and obtain
  \begin{align*}
    2^{-k} \sum_{j=0}^{\lfloor k/2 \rfloor - t} \binom{k}{j}
     &\,\geq\, \frac{t}{2 \sqrt{\lceil k/2 \rceil}}
          \,\begin{cases}
              \displaystyle
              \exp\Bigl[- \frac{8 \, (\log 2) \, t \, (t - 1/2)
                              }{\lceil k/2 \rceil + 1}
                  \Bigr]
                &\quad\text{for odd~$k$,}\\
              \displaystyle
              \exp\Bigl[- \frac{8 \, (\log 2) \, (t-1/2)^2
                              }{k/2 + 1}
                  \Bigr]
                &\quad\text{for even~$k$.}\\
            \end{cases}
  \end{align*}
  For \mbox{$t \geq \sqrt{\lceil k/2 \rceil}/(1+2/\sqrt{\pi})$},
  the prefactor simplifies as stated in the claimed inequality.
\end{proof}

\begin{lem} \label{lem:BinomTail2}
  For $k,k'\in\N_0$ we have for all $k\geq k'$ that
  \begin{equation*}
    2^{-k} \Biggl[\sum_{j=0}^{\left\lfloor \frac{k-k'}{2}
                              \right\rfloor}
                   \binom{k}{j}
                  + \sum_{j=\left\lceil \frac{k+k'+1}{2} \right\rceil}^k
                     \binom{k}{j}
            \Biggr]  
      \,\geq\, 2^{-k'} \,.    
  \end{equation*}
\end{lem}
\begin{proof}
  The proof follows by induction over $k\geq k'$. A speciality here is that
  in the induction step we assume the statement for $k$ and prove it for $k+2$, 
  which is sufficient when the base case is verified for $k=k'$ 
  and $k=k'+1$. 

  For $k=k'$ and $k=k'+1$ we have $2^{-k'} \binom{k}{0}$
  and $2^{-(k'+1)}[\binom{k'}{0}+\binom{k'+1}{k'+1}]$
  which proves the inequality. (We even have equality.)
  
  For the induction step from $k$ to $k+2$ where $k\geq k'$,
  via Pascal's rule, as well as using 
  $\binom{k}{\left \lfloor \frac{k+2-k'}{2} \right \rfloor}
    \geq \binom{k}{\left \lfloor \frac{k-k'}{2} \right \rfloor}$,
  we obtain
  \begin{equation*}
    \sum_{j=0}^{\left\lfloor \frac{k+2-k'}{2}\right\rfloor}\binom{k+2}{j}
      \,=\, 4 \sum_{j=0}^{\left\lfloor \frac{k-k'}{2}\right\rfloor-1}\binom{k}{j}
            + 3 \, \binom{k}{\left\lfloor \frac{k-k'}{2}\right\rfloor}
            + \binom{k}{\left \lfloor \frac{k+2-k'}{2} \right \rfloor}
    \,\geq\, 4 \sum_{j=0}^{\left\lfloor \frac{k-k'}{2}\right\rfloor}\binom{k}{j}.
  \end{equation*}
  Similarly, with  
  $\binom{k}{\left \lfloor \frac{k+k'+1}{2} \right \rfloor}
    \geq \binom{k}{\left \lfloor \frac{k+k'+3}{2} \right \rfloor}$,
  one can show
  \begin{equation*}
    \sum_{j=\left\lceil \frac{k+k'+3}{2} \right\rceil}^{k+2}
        \binom{k+2}{j}
      \,\geq\, 4 \sum_{j=\left\lceil \frac{k+k'+1}{2} \right\rceil}^k
                \binom{k}{j}.
  \end{equation*}
Now, by the induction hypothesis the assertion is proven. 
\end{proof}

\section*{Acknowledgements}

The authors wish to express their gratitude to Erich Novak
for many detailed hints and discussions during the work on this paper.
We also wish to thank Glenn Byrenheidt, Stefan Heinrich, Lutz K\"ammerer,
David Krieg, and Mario Ullrich for their advice.
Daniel Rudolf gratefully acknowledges support of
the Felix-Bernstein-Institute for Mathematical Statistics in the Biosciences (Volkswagen Foundation) and the Campus laboratory AIMS.

\end{document}